\documentclass[reqno]{amsart}

\usepackage{amsmath,amssymb,amsthm,mathrsfs}

\theoremstyle{plain}
\numberwithin{equation}{section}

\newcommand{\BC}{{\mathbb C}}\newcommand{\BD}{{\mathbb D}}

\newcommand{\BN}{{\mathbb N}}

\newcommand{\BT}{{\mathbb T}}

\newcommand{\cB}{{\mathcal B}}

\newcommand{\cF}{{\mathcal F}}
\newcommand{\cG}{{\mathcal G}}\newcommand{\cH}{{\mathcal H}}

\newcommand{\cK}{{\mathcal K}}\newcommand{\cL}{{\mathcal L}}

\newcommand{\cV}{{\mathcal V}}
\newcommand{\cW}{{\mathcal W}}\newcommand{\cX}{{\mathcal X}}
\newcommand{\cY}{{\mathcal Y}}\newcommand{\cZ}{{\mathcal Z}}

\newcommand{\fC}{{\mathfrak C}}

\newcommand{\fI}{{\mathfrak I}}

\newcommand{\wtilE}{\widetilde{E}}\newcommand{\wtilF}{\widetilde{F}}
\newcommand{\wtilH}{\widetilde{H}}

\newcommand{\wtilT}{\widetilde{T}}

\newcommand{\whatE}{\widehat{E}}

\newcommand{\whatU}{\widehat{U}}

\newcommand{\al}{\alpha}
\newcommand{\be}{\beta}

\newcommand{\De}{\Delta}

\newcommand{\la}{\lambda}

\newcommand{\im}{\textup{Im\,}}

\newcommand{\kr}{\textup{Ker\,}}

\newcommand{\mat}[2]{\ensuremath{\left[\begin{array}{#1}#2\end{array} \right]}}

\newcommand{\sbm}[1]{\left[\begin{smallmatrix} #1\end{smallmatrix}\right]}

\newcommand{\tu}[1]{\textup{#1}}
\newcommand{\wtil}[1]{{\widetilde{#1}}}
\newcommand{\what}[1]{{\widehat{#1}}}

\newcommand{\ands}{\quad\mbox{and}\quad}

\theoremstyle{plain}
\newtheorem{theorem}{Theorem}[section]
\newtheorem{corollary}[theorem]{Corollary}
\newtheorem{lemma}[theorem]{Lemma}
\newtheorem{proposition}[theorem]{Proposition}
\theoremstyle{definition}

\newtheorem{remark}[theorem]{Remark}

\newcommand{\Up}{J}

\newcommand{\Ind}{\textup{Ind}}

\begin{document}

\title{Equivalence after extension and Schur coupling coincide for inessential operators}

\author[S. ter Horst]{S. ter Horst}
\address{S. ter Horst, Department of Mathematics, Unit for BMI, North-West
University, Potchefstroom, 2531 South Africa}
\email{Sanne.TerHorst@nwu.ac.za}

\author[M. Messerschmidt]{M. Messerschmidt}
\address{M. Messerschmidt, Department of Mathematics and Applied Mathematics; University of Pretoria; Private bag X20 Hatfield; 0028 Pretoria; South Africa}
\email{mmesserschmidt@gmail.com}

\author[A.C.M. Ran]{A.C.M. Ran}
\address{A.C.M. Ran, Department of Mathematics, FEW, VU university Amsterdam, De Boelelaan 1081a, 1081 HV Amsterdam, The Netherlands and Unit for BMI, North-West~University, Potchefstroom, South Africa}
\email{a.c.m.ran@vu.nl}

\author[M. Roelands]{M. Roelands}
\address{M. Roelands, Department of Mathematics, Unit for BMI, North-West
University,
Potchefstroom, 2531 South Africa}
\email{mark.roelands@gmail.com}

\author[M. Wortel]{M. Wortel}
\address{M. Wortel, Department of Mathematics, Unit for BMI, North-West
University, Potchefstroom, 2531 South Africa
and DST-NRF Centre of Excellence in Mathematical and Statistical Sciences (CoE-MaSS).}
\email{marten.wortel@gmail.com}

\thanks{This work is based on the research supported in part by the National Research Foundation of South Africa (Grant Numbers 90670 and 93406). Messerschmidt's research was funded by The Claude Leon Foundation, registered in South Africa No.\ P436/63. Wortel's research was supported in part by the DST-NRF Centre of Excellence in Mathematical and Statistical Sciences (CoE-MaSS), grant reference number BA2017/260.}

\subjclass[2010]{Primary 47A62; Secondary 47B07, 47A53, 47B06}

\keywords{Equivalence after extension; Schur coupling; inessential operators;
compact operators; Fredholm operators}

\begin{abstract}
In recent years the coincidence of the operator relations equivalence after extension (EAE) and Schur coupling (SC) was settled for the Hilbert space case. For Banach space operators, it is known that SC implies EAE, but the converse implication is only known for special classes of operators, such as Fredholm operators with index zero and operators that can in norm be approximated by invertible operators. In this paper we prove that the implication EAE $\Rightarrow$ SC also holds for inessential Banach space operators.
The inessential operators were introduced as a generalization of the compact operators, and include, besides the compact operators, also the strictly singular and strictly co-singular operators; in fact they form the largest ideal such that the invertible elements in the associated quotient algebra coincide with (the equivalence classes of) the Fredholm operators.
\end{abstract}

\maketitle

\section{Introduction}
\setcounter{equation}{0}

Throughout this paper, let $U\in\cB(\cX)$ and $V\in\cB(\cY)$ be two Banach space operators. Here $\cB(\cV,\cW)$ stands for the Banach space of bounded linear operators from the Banach space $\cV$ into the Banach space $\cW$, abbreviated to $\cB(\cV)$ if $\cV=\cW$. The term operator will always mean bounded linear operator, and invertibility of an operator will imply that the inverse is bounded as well.

The operators $U$ and $V$ are said to be {\em equivalent after extension (EAE)} when there exist Banach spaces $\cX_0$ and $\cY_0$ such that $U\oplus I_{\cX_0}$ and $V\oplus I_{\cY_0}$ are equivalent, that is, when there exist invertible operators $E$ in $\cB(\cY\oplus\cY_0, \cX\oplus\cX_0)$ and $F$ in $\cB(\cX\oplus\cX_0,\cY\oplus\cY_0)$ such that
\begin{equation}\label{EAE}
\mat{cc}{U&0\\0&I_{\cX_0}}=E\mat{cc}{V&0\\0& I_{\cY_0}}F.
\end{equation}
We will encounter the special case
where $\cX_0=\{0\}$ or $\cY_0=\{0\}$, in which case we say that $U$ and $V$ are {\em equivalent after one-sided extension (EAOE)}. Furthermore, we say that $U$ and $V$ are {\em Schur coupled (SC)} whenever there exists an operator $M=\sbm{A&B\\C&D}\in\cB(\cX\oplus\cY,\cX\oplus\cY)$ with $A$ and $D$ invertible and
\[
U=A-BD^{-1}C \ands V=D-CA^{-1}B.
\]
Hence $U$ and $V$ are the Schur complements of the block operator matrix $M$ with respect to $D$ and $A$, respectively. We also recall that $U$ and $V$ are called {\em matricially coupled (MC)} whenever there exists an invertible operator $\whatU\in \cB(\cX\oplus\cY)$ such that
\[
\whatU=\mat{cc}{U & *\\ * & *}\ands \whatU^{-1}=\mat{cc}{* & *\\ * & V}.
\]
These operator relations originated in the study of certain Wiener-Hopf integral operators, in the seminal paper \cite{BGK84}, cf., \cite[Chapter XIII]{GGK90}, and have since found many other applications, for instance in the study of diffraction theory \cite{CDS14,S17}, truncated Toeplitz operators \cite{CP16,CParxiv} and unbounded operator functions \cite{ETarxiv}, to mention only a few recent references. Operator extensions into block operator matrices and their spectral theory were studied in detail in \cite{T08}. Schur complements have many applications, for instance in matrix function factorizations \cite{BGKR08}, see also \cite{Z05} for its role in matrix analysis and applications in statistics and numerical analysis.

One of the main features in \cite{BGK84} is the fact that MC implies EAE, cf.,
\cite[Section III.4]{GGK90}, which is used to derive Fredholm properties of the integral operators in question.
The converse implication, that EAE implies MC was proven in \cite{BT92a},
while in \cite{BT92b}, see also  \cite{BGKR05}, it was shown that SC implies both MC and EAE. Thus the operator relations EAE and MC coincide, and both are implied by SC. In \cite{BT94} the question was raised whether all three operator relation may be the same, hence, whether EAE=MC also implies SC. This matter has not been settled yet, though significant progress was made in recent years. It is for instance true that EAE=MC implies SC if SC is an equivalence relation \cite{BGKR05}, as is the case for EAE and MC, and affirmative answers have been obtained for specific classes of operators such as Fredholm operators with index 0 (and without the index condition for Hilbert space operators) \cite{BGKR05}, and more recently, for Banach space operators that can be approximated in norm by invertible operators, leading to an affirmative answer for operators on separable Hilbert spaces \cite{tHR13}, and for general Hilbert space operators in \cite{T14}.

A Banach space operator $T\in\cB(\cV,\cW)$ is called {\em inessential} if for any $S\in\cB(\cW,\cV)$ the operator $I-ST$ is Fredholm, or, equivalently, if for any $S\in\cB(\cW,\cV)$ the operator $I-TS$ is Fredholm, cf., \cite[Section 7.1]{A04}. We write $\fI(\cV,\cW)$ for the set of inessential operators in $\cB(\cV,\cW)$, abbreviated to $\fI(\cV)$ if $\cV=\cW$.

The following theorem is the main result of the present paper.

\begin{theorem}\label{T:main1}
Let $U\in\fI(\cX)$ and $V\in\fI(\cY)$. Then
\[
\mbox{$U$ and $V$ are EAE}\quad \Longleftrightarrow \quad
\mbox{$U$ and $V$ are EAOE}\quad \Longleftrightarrow \quad
\mbox{$U$ and $V$ are SC}.
\]
Moreover, in this case $U$ and $V$ are EAE and EAOE with extensions on finite dimensional Banach spaces.
\end{theorem}

The inessential operators were introduced by Kleinecke in \cite{K63} as a generalization of the compact operators. In fact, the set $\fI(\cV)$ of inessential operators in $\cB(\cV)$ forms the largest (two-sided) closed ideal in $\cB(\cV)$ for which the invertible elements in the quotient $\cB(\cV)/\fI(\cV)$ coincide with (the equivalence classes of) the Fredholm operators in $\cB(\cV)$, as observed by Schlechter \cite{S68}. Hence, besides the compact operators it also includes the strictly singular and strictly co-singular operators on $\cV$; see \cite[Section 7.1]{A04} for further details and references. By the Pitt-Rosenthal Theorem, for $1\leq p<q<\infty$ all operators in $\cB(\ell^q,\ell^p)$ are compact; this is not the case for $\cB(\ell^p,\ell^q)$, however all operators in $\cB(\ell^p,\ell^q)$ are strictly singular, c.f., \cite{LT77}. Also, strictly singular operators can have non-separable ranges \cite{GT63}.

Although the inessential operators typically form a small class of operators, there are `exotic' Banach spaces for which the above result together with the observation from \cite{BGKR05} that EAE implies SC for Fredholm operators with index zero provides a confirmative answer to the question whether EAE implies SC. These spaces are referred to as Banach spaces with (very) few operators, c.f., \cite{M03} and the references therein. An infinite dimensional Banach space $\cV$ is said to have {\em few operators} if any operator is of the form $\la I_\cV + S$ with $\la$ a scalar and $S$ strictly singular, while $\cV$ has {\em very few} operators if one replaces strictly singular by compact. In both cases, the Calkin algebra has dimension 1. Although all `classical' Banach spaces are not of these types, all hereditarily indecomposable Banach spaces have few operators \cite{GM93}, and there are examples which are reflexive, separable and have other reasonable Banach space properties, c.f., \cite[Subsection 6.3]{G07}. A construction of a Banach space with very few operators was first given in 2011 by Argyros and Heydon \cite{AH11}. Such spaces are relevant to the invariant subspace problem, since each operator on it has a nontrivial invariant subspace; that this is in general not the case for Banach spaces with few operators was shown by Read  \cite{R99}.

For Hilbert space operators, as well as many Banach spaces, inessential operators are also compact. See \cite[Section 6]{T14} for comments on the implications EAE $\Rightarrow$ EAOE $\Rightarrow$ SC for the case of compact Hilbert operators, including the observation that the extension in EAOE can be taken finite dimensional. The proofs in \cite{T14}, for the general Hilbert space case, rely strongly on the use of spectral resolutions, and do not carry over to the Banach space case. In fact, while in the Hilbert space case the proof of EAE $\Rightarrow$ SC goes via EAOE, as in the current paper, Proposition 4.3 in \cite{tHMR15}, shows SC cannot imply EAOE in general, since invertible operators $U\in\cB(\ell^p)$ and $V\in\cB(\ell^q)$, for $1\leq p\neq q<\infty$, are SC but cannot be EAOE.

\begin{remark}
The results of in the final section of \cite{tHMR15} (i.e., Theorems 5.2, 5.3, Corollaries 5.4, 5.7 and Proposition 5.6) were proved using the notion of essentially incomparable Banach spaces assuming one of the operators, $U$ or $V$, is compact. In hindsight we note that all these results can easily be extended to the case where one of the operators in inessential, and, in particular, they remain true if one of the operators is strictly singular or strictly co-singular. Indeed, the proofs of these results from \cite{tHMR15} rely purely on the fact that the invertible elements in the Calkin algebra of the compact operators correspond to the Fredholm operators, which is also true if the compact operators are replaced by the inessential operators.
\end{remark}

As indicated above, to prove Theorem \ref{T:main1} it remains to show that EAE of $U$ and $V$ implies they are also EAOE. A large portion of the proof can be conducted at a slightly more general level, namely in the case when the operators $E$ and $F$ that establish the EAE of $U$ and $V$ (in the special form of \cite[Corollary 4.2]{tHR13}) have a left upper corner, $E_{11}$, and right lower corner, $F_{22}$, respectively, that have complementable kernels and ranges, i.e., they are relatively regular \cite[Theorem 7.2]{A04}. This special case will be considered in Section \ref{S:comp}. In Section \ref{S:ProofMain} we return to the case that $U$ and $V$ are inessential and show that in that case the operators $E_{11}$ and $F_{22}$ are Fredholm, leading to a proof of the main result. In the final section we discuss an application of a result from \cite{tHMR15} to multiplication operators on $L^p$-spaces.

We conclude this introduction with a few words about notation and terminology. Most of the notation is standard, or already explained above. For a complementable subspace $\cF$ of a Banach space $\cZ$, say with complement $\cG$, we write $P_\cF$ for the projection onto $\cF$ along $\cG$ viewed as an operator on $\cZ$. The complementary space is not include in the notation, since for all the projections in this paper it will be clear what the complement in question is. Occasionally we want to view the projection as mapping into $\cF$, in which case we use the notation $\Pi_\cF$, hence $P_\cF\in\cB(\cZ,\cZ)$, while $\Pi_\cF\in\cB(\cZ,\cF)$. The embedding of $\cF$ into $\cH$ will be denoted by $\Up_\cF\in\cB(\cF,\cZ)$.

\section{The case that $E_{11}$ and $F_{22}$ have complementable kernels and ranges.}
\label{S:comp}

By Corollary 4.2 of \cite{tHR13}, if $U$ and $V$ are EAE one may without loss of generality assume $\cX_0=\cY$ and $\cY_0=\cX$ with the right upper corners of $F$ and $F^{-1}$ equal to $I_\cY$ and $I_\cX$, respectively. The following lemma provides an extension of this result.

\begin{lemma}\label{L:EFform}
Assume that $U\in\cB(\cX)$ and $V\in\cB(\cY)$ are EAE. Then we may assume \eqref{EAE} holds with $\cX_0=\cY$, $\cY_0=\cX$ and with $F$ and $E$ and their inverses of the following operator block form
\begin{equation}\label{EFform1}
\begin{aligned}
F=\mat{cc}{F_{11}& I_\cY\\ F_{21} & F_{22}},&\quad
E=\mat{cc}{E_{11} & U\\ E_{21} & -F_{11}}\\
F^{-1}=\mat{cc}{-F_{22}& I_\cX\\ I+F_{11}F_{22} & -F_{11}},&\quad
E^{-1}=\mat{cc}{\what{E}_{11} & V\\ \what{E}_{21} & F_{22}}.
\end{aligned}
\end{equation}
Moreover, in this case the following identities hold:
\begin{align}
&(i)\ I=F_{21} - F_{22}F_{11},
\quad (ii)\ U=E_{11}VF_{11}+UF_{21},
\quad (iii)\ E_{21}VF_{11}=F_{11}F_{21},\notag\\
&(iv)\ E_{11}V=-UF_{22},
\quad (v)\ F_{11}F_{22}=E_{21}V-I,
\quad (vi)\ \whatE_{11}U=V F_{11}, \label{EAEIds} \\
& (vii)\ \whatE_{21}U=F_{21},
\quad (viii)\ E_{11}\whatE_{11}=I-U \whatE_{21},
\quad (ix)\ E_{21}\whatE_{11}=F_{11}\whatE_{21},\notag\\
& (x)\ \whatE_{11}E_{11}=I-VE_{21},
\quad (xi)\ \whatE_{21} E_{11}=-F_{22} E_{21}.\notag
\end{align}
\end{lemma}

\begin{proof}[\bf Proof]
By Corollary 4.2 in \cite{tHR13} the identity \eqref{EAE} holds with $\cX_0=\cY$, $\cY_0=\cX$ and with the right upper corners of $F$, $E$, $F^{-1}$ and $E^{-1}$ as in \eqref{EFform1}. Hence $F$ is of the form as in \eqref{EFform1}. Next we derive the formula for $F^{-1}$. Since $F$ is invertible, as well as the right upper corner $I_\cY$ of $F$, the Schur complement with respect to $I_\cY$, namely $\De=F_{21}-F_{22}F_{11}$ is also invertible, and we can employ the standard inversion formula, cf., \cite[pp.\ 28--29]{BGKR08}, to find that
\[
F^{-1}=\mat{cc}{-\De^{-1}F_{22} & \De^{-1}\\ I+F_{11}\De^{-1}F_{22} & -F_{11}\De^{-1}}.
\]
However, the right upper corner $\De^{-1}$ of $F^{-1}$ is equal to $I_\cX$. Inserting $\De^{-1}=I_\cX$ in the formula for $F^{-1}$ gives the formula for $F^{-1}$ in \eqref{EFform1}. Moreover, the identity $\De=I_\cX$ gives identity (i) in \eqref{EAEIds}. We have
\begin{equation}\label{EAEmatids}
\begin{aligned}
&\mat{cc}{U&0\\0&I_{\cY}}=E\mat{cc}{V&0\\0& I_{\cX}}F,\quad
\mat{cc}{U&0\\0&I_{\cY}}F^{-1}=E\mat{cc}{V&0\\0& I_{\cX}},\\
&\qquad\qquad\qquad\qquad E^{-1}\mat{cc}{U&0\\0&I_{\cY}}=\mat{cc}{V&0\\0& I_{\cX}}F.
\end{aligned}
\end{equation}
That the right lower corners of $E$ and $E^{-1}$ are given by $-F_{11}$ and $F_{22}$, respectively, follows from an inspection of the right lower corners of the last two identities in \eqref{EAEmatids}. The first identity of \eqref{EAEmatids} gives identities (ii)--(v) of \eqref{EAEIds}, while the third identity of \eqref{EAEmatids} gives identities (vi) and (vii) of \eqref{EAEIds}. Finally, writing out $E E^{-1}=I$ gives identities (viii) and (ix) of \eqref{EAEIds} and $E^{-1}E=I$ gives identities (x) and (xi) of \eqref{EAEIds}.
\end{proof}

For the remainder of this section we will assume that the operators $E_{11}$ and $F_{22}$ in the decompositions of $E$ and $F$ in \eqref{EFform1} have complementable kernels and ranges, in particular, their ranges are assumed to be closed. In this case $E_{11}$ and $F_{22}$ have decompositions of the following form:
\begin{equation}\label{E11F22Dec}
\begin{aligned}
& F_{22}=\mat{cc}{F_{22}'&0\\0&0}\colon\mat{c}{\cK_{2}\\ \kr F_{22}}\to\mat{c}{\im F_{22}\\ \cH_2},\\
& E_{11}=\mat{cc}{E_{11}'&0\\0&0}\colon\mat{c}{\cF_{1}\\ \kr E_{11}}\to\mat{c}{\im E_{11}\\ \cG_1}.
\end{aligned}
\end{equation}
Moreover, since $E_{11}$ and $F_{22}$ have closed range, $E_{11}'$ and $F_{22}'$ are bounded bijections, and hence they are invertible operators.

The next proposition is the main result of the present section.

\begin{proposition}\label{P:ComplRed}
Let $U\in\cB(\cX)$ and $V\in\cB(\cY)$ be Banach space operators that are EAE. Let $E$ and $F$ be as in Lemma \ref{L:EFform} and assume $E_{11}$ and $F_{22}$ have complementable kernels and ranges, so that they admit decompositions as in \eqref{E11F22Dec}. Then $U$ and $V$ are also EAE with $\cX_0=\kr E_{11}$ and $\cY_0=\cH_2$. Moreover, if $\kr E_{11}$ is isomorphic to a complemented subspace of $\cH_2$ or $\cH_2$ is isomorphic to a complemented subspace of $\kr E_{11}$, then $U$ and $V$ are EAOE and hence SC.
\end{proposition}

In order to prove this result we first prove a few lemmas. In the remainder of this section $U$ and $V$ will be EAE as in Lemma \ref{L:EFform} and we assume $E_{11}$ and $F_{22}$ have complementable kernels and ranges, so that they admit decompositions as in \eqref{E11F22Dec} with $E_{11}'$ and $F_{22}'$ invertible.

\begin{lemma}\label{L:UVform}
The operators $U$ and $V$ admit operator matrix representations of the form
\begin{equation}\label{UVform}
\begin{aligned}
& U=\mat{cc}{U_{11}&U_{12}\\0&U_{22}}\colon\mat{c}{\im F_{22}\\ \cH_2}\to\mat{c}{\im E_{11}\\ \cG_1},\\
& V=\mat{cc}{V_{11}&0\\V_{21}&V_{22}}\colon\mat{c}{\cK_2\\ \kr F_{22}}\to\mat{c}{\cF_1\\ \kr E_{11}}.
\end{aligned}
\end{equation}
Moreover, $V_{11}$ and $U_{11}$ are equivalent, more specifically, $E_{11}' V_{11}=-U_{11} F_{22}'$.
Finally, $\Pi_{\kr F_{22}}E_{21}\Up_{\kr E_{11}}$ is a left inverse of $V_{22}$ and $\Pi_{\cH_2}\whatE_{21}\Up_{\cG_1}$ is a right inverse of $U_{22}$, so that
 $\im U_{22}=\cG_1$ and $V_{22}$ is bounded below.
\end{lemma}


\begin{proof}[\bf Proof]
Decompose $U$ and $V$ as above, writing $U_{21}$ for the left lower corner of $U$ and $V_{12}$ for the right upper corner of $V$. From identity (iv) in \eqref{EAEIds} we then obtain that
\begin{align*}
\mat{cc}{E_{11}' V_{11} & E_{11}' V_{12}\\ 0&0} &=\mat{cc}{E_{11}'&0\\ 0&0}\mat{cc}{V_{11}& V_{12}\\ V_{21}&V_{22}}=E_{11}V= \\
& =-UF_{22}=\mat{cc}{U_{11}&U_{12}\\ U_{21} & U_{22}}\mat{cc}{-F_{22}' &0\\ 0&0}
=\mat{cc}{-U_{11}F_{22}' & 0\\ -U_{21}F_{22}'&0}.
\end{align*}
Inspecting the four identities and using that $F_{22}'$ and $E_{11}'$ are invertible we obtain that $V_{12}=0$, $U_{21}=0$ and $E_{11}' V_{11}=-U_{21}F_{22}'$, as claimed. In particular, $U$ and $V$ are as in \eqref{UVform}.

Since $U$ is as in \eqref{UVform} we have $\Pi_{\cG_1}U=U_{22}\Pi_{\cH_2}$. Using that $\Pi_{\cG_1} E_{11}=0$, we obtain from identity (viii) that
\[
0=\Pi_{\cG_1}E_{11}\whatE_{11}\Up_{\cG_1}=\Pi_{\cG_1}(I-U\whatE_{21})\Up_{\cG_1}=I_{\cG_1}-U_{22}\Pi_{\cH_2}\whatE_{21}\Up_{\cG_1}.
\]
Hence $\Pi_{\cH_2}\whatE_{21}\Up_{\cG_1}$ is a right inverse of $U_{22}$.

Since $V$ is as in \eqref{UVform} we find that $V \Up_{\kr F_{22}}=\Up_{\kr E_{11}}V_{22}$. Thus, by identity (v) in \eqref{EAEIds} we have
\begin{align*}
0
&= \Pi_{\kr F_{22}} F_{11}F_{22}\Up_{\kr F_{22}}=\Pi_{\kr F_{22}}(E_{21}V-I)\Up_{\kr F_{22}} \\
 &=\Pi_{\kr F_{22}}E_{21}\Up_{\kr E_{11}}V_{22}-I_{\kr F_{22}}.
\end{align*}
Thus $\Pi_{\kr F_{22}}E_{21}\Up_{\kr E_{11}}$ is a left inverse of $V_{22}$.
\end{proof}


In case $U_{22}$ and $V_{22}$ are invertible, we arrive at the conclusion of Proposition \ref{P:ComplRed}.

\begin{lemma}\label{L:InvSuf}
Assume $U_{22}$ and $V_{22}$ in the decompositions of Lemma \ref{L:UVform} are invertible. Then $U\oplus I_{\kr E_{11}}$ and $V\oplus I_{\cH_2}$ are equivalent, i.e., \eqref{EAE} holds with $\cX_0=\kr E_{11}$ and $\cY_0=\cH_2$. Furthermore, if $\kr E_{11}$ is isomorphic to a complemented subspace of $\cH_2$ or $\cH_2$ is isomorphic to a complemented subspace of $\kr E_{11}$, then $U$ and $V$ are EAOE and hence SC.
\end{lemma}

\begin{proof}[\bf Proof]
Assume $U_{22}$ and $V_{22}$ are invertible. Then we can factor $U$ and $V$ as
\[
U=\mat{cc}{I_{\im E_{11}}&U_{12}\\ 0&U_{22}}\mat{cc}{U_{11}&0\\0& I_{\cH_2}}
\ands
V=\mat{cc}{V_{11}&0\\0&I_{\kr E_{11}}}\mat{cc}{I_{\cK_2}&0\\V_{21}&V_{22}}.
\]
By our assumption, the left hand side in the factorization of $U$ and the right hand side in the factorization of $V$ are invertible. Thus we obtain that $U$ and $U_{11}\oplus I_{\cH_2}$ are equivalent and $V$ and $V_{11}\oplus I_{\kr E_{11}}$ are equivalent. We may therefore replace $U$ by $U_{11}\oplus I_{\cH_2}$ and $V$ by $V_{11}\oplus I_{\kr E_{11}}$. Recall from Lemma \ref{L:UVform} that $U_{11}$ and $V_{11}$ are equivalent via $E_{11}' V_{11}=-U_{22} F_{22}'$. One simply verifies that
\begin{align*}
&\mat{cc|c}{U_{11}&0&0\\ 0&I_{\cH_2} &0 \\ \hline 0& 0& I_{\kr E_{11}}}
=\mat{cc|c}{E_{11}'&0&0\\ 0&0 &I_{\cH_2} \\ \hline 0& I_{\kr E_{11}} & 0} \times\\
&\qquad\qquad\qquad \times\mat{cc|c}{V_{11}&0&0\\ 0&I_{\kr E_{11}} &0 \\ \hline 0& 0& I_{\cH_2}}
\mat{cc|c}{-(F_{22}')^{-1}&0&0\\ 0&0 &I_{\kr E_{11}} \\ \hline 0& I_{\cH_2}& 0}.
\end{align*}
from which we conclude that $U_{11}\oplus I_{\cH_2}$ and $V_{11}\oplus I_{\kr E_{11}}$ are EAE with extensions $\cX_0=\kr E_{11}$ and $\cY_0=\cH_2$, and hence the same holds for $U$ and $V$.

Assume there exists an injective operator $T\in\cB(\cH_2,\kr E_{11})$ whose range $\cZ:=\im T$ is complementable in $\kr E_{11}$, and hence closed. Hence $T$ admits a left inverse $T^+\in \cB(\kr E_{11},\cH_2)$. Let $\cZ'$ be a complement of $\cZ$ in $\kr E_{11}$. Then
\begin{align*}
&\mat{ccc}{U_{11}&0&0\\ 0&I_{\cH_2} &0 \\ 0& 0& I_{\cZ'}}=\\
&\qquad\quad=\mat{ccc}{E_{11}'&0&0\\ 0&T^{+} \Up_{\cZ} &0 \\ 0& 0& I_{\cZ'}}
\mat{ccc}{V_{11}&0&0\\ 0&I_{\cZ}& 0\\ 0&0& I_{\cZ'}}
 \mat{ccc}{-(F_{22}')^{-1}&0&0\\ 0&\Pi_{\cZ}T &0\\ 0&0 &I_{\cZ'}}.
\end{align*}
Note that the operators left and right of $V_{11}\oplus I_{\kr E_{11}}$ are invertible. Hence $U_{11}\oplus I_{\cH_2}$ and $V_{11}\oplus I_{\kr E_{11}}$ are EAOE, and the same holds true for $U$ and $V$. In case $\kr E_{11}$ can be embedded into $\cH_2$, a similar argument applies.
\end{proof}

In the remainder of this section we show that $U_{22}$ and $V_{22}$ are invertible.

\begin{lemma}\label{L:EFadj}
Without loss of generality we have
\begin{equation}\label{EFadj}
E_{21}=P_{\kr F_{22}}E_{21}
\ands F_{21}=P_{\cH_2}.
\end{equation}
\end{lemma}

\begin{proof}[\bf Proof]
Write $F_{22}^+$ for the Moore-Penrose generalized inverse of $F_{22}$, that is,
\[
F_{22}^+=\mat{cc}{(F_{22}')^{-1}&0\\0&0}\colon \mat{c}{\im F_{22}\\ \cH_2}\to\mat{c}{\cK_2\\ \kr F_{22}}.
\]
Now set $X=F_{22}^+ \whatE_{21}$ and define $\wtilE=\sbm{I&0\\X&I}E$ and $\wtilF=F\sbm{I&0\\ -XU &I}$, that is,
\[
\wtilE=\mat{cc}{E_{11}&U\\ E_{21}+XE_{11} & -F_{11}+XU}\ands
\wtilF=\mat{cc}{F_{11}-XU & I\\ F_{21}-F_{22}XU & F_{22}}.
\]
Note that this transformation does not effect the $(1,1)$-entry of $E$, nor the $(2,2)$-entry of $F$, hence the decompositions of $U$ and $V$ in Lemma \ref{L:UVform} remain unchanged. Moreover, by Lemma 4.3 in \cite{tHR13} the EAE of $U$ and $V$ is also established through $\wtilE$ and $\wtilF$, i.e., \eqref{EAE} holds with $E$ and $F$ replaced by $\wtilE$ and $\wtilF$, respectively, and $\wtilE$ and $\wtilF$ still have the special form described in Lemma \ref{L:EFform}. Hence, without loss of generality, we started with $\wtilE$ and $\wtilF$ instead of $E$ and $F$. Using identity (xi) in \eqref{EAEIds} we obtain that the $(2,1)$-entry in $\wtilE$ is given by
\begin{align*}
\wtilE_{21}&=E_{21}+XE_{11}
=E_{21}+F_{22}^+\whatE_{21}E_{11}
=E_{21}-F_{22}^+ F_{22}E_{21}\\
&=E_{21}-(I-P_{\kr F_{22}})E_{21}=P_{\kr F_{22}}E_{21}.
\end{align*}
So $P_{\kr F_{22}}\wtilE_{21}=\wtilE_{21}$. Since the $(2,2)$-entry of $F$ is not effected by the transformation we see that without loss of generality we may assume $E_{21}=P_{\kr F_{22}}E_{21}$. Next, using identities (vii) and (i), observe that the $(2,1)$-entry of $\wtilF$ is given by
\begin{align*}
\wtilF_{21}&=F_{21}-F_{22}XU
=F_{21}-F_{22}F_{22}^+ \whatE_{21}U=F_{21}- P_{\im F_{22}} F_{21}
=(I- P_{\im F_{22}}) F_{21}\\
&= (I- P_{\im F_{22}}) (I-F_{11}F_{22})=(I- P_{\im F_{22}})=P_{\cH_2}.
\end{align*}
Again using that the $(2,2)$-entry in $F$ is not effected by the transformation we see that without loss of generality $F_{21}=P_{\cH_2}$.
\end{proof}

\begin{proof}[\bf Proof of Proposition \ref{P:ComplRed}] We may assume \eqref{EFadj} hold. By Lemma \ref{L:InvSuf} it suffices to show that $U_{22}$ and $V_{22}$ in the decompositions of $U$ and $V$ from Lemma \ref{L:UVform} are invertible. In fact, it suffices to show the left inverse of $V_{22}$ and right inverse of $U_{22}$ obtained in the the last claim of Lemma \ref{L:UVform} are also a right inverse, respectively, left inverse, i.e., it remains to show that
\begin{equation}\label{RemIds}
V_{22}\Pi_{\kr F_{22}}E_{21}\Up_{\kr E_{11}}=I_{\kr E_{11}}
\ands
\Pi_{\cH_2}\whatE_{21}\Up_{\cG_1}U_{22}=I_{\cH_2}.
\end{equation}
Using $E_{21}=P_{\kr F_{22}}E_{21}$ together with identity (x) in \eqref{EAEIds} gives the first identity:
\begin{align*}
0&
=\Pi_{\kr E_{11}}\whatE_{11}E_{11}\Up_{\kr E_{11}}
=I_{\kr E_{11}}-\Pi_{\kr E_{11}}VE_{21}\Up_{\kr E_{11}}\\
&=I_{\kr E_{11}}-\Pi_{\kr E_{11}}V P_{\kr F_{22}}E_{21}\Up_{\kr E_{11}}
=I_{\kr E_{11}}-V_{22} \Pi_{\kr F_{22}}E_{21}\Up_{\kr E_{11}}.
\end{align*}
Since $E_{21}=P_{\kr F_{22}}E_{21}$, we have $F_{22}E_{21}=0$ and thus
\[
0=-F_{22}E_{21}\Up_{\cF_1}=\whatE_{21}E_{11}\Up_{\cF_1}=\whatE_{21}\Up_{\im E_{11}} E_{11}'.
\]
From this we obtain that $\whatE_{21}P_{\im E_{11}}=0$. In other words $\whatE_{21}=\whatE_{21}P_{\cG_1}$. The second identity in \eqref{RemIds} now follows using  $F_{21}=P_{\cH_2}$ and identity (vii) in \eqref{EAEIds}:
\[
I_{\cH_2}=\Pi_{\cH_2}F_{21}\Up_{\cH_2}
=\Pi_{\cH_2}\whatE_{21}U\Up_{\cH_2}
=\Pi_{\cH_2}\whatE_{21}P_{\cG_1}U\Up_{\cH_2}
=\Pi_{\cH_2}\whatE_{21}\Up_{\cG_1}U_{22}.\qedhere
\]
%
\end{proof}

\section{Proof of the main result}\label{S:ProofMain}

In this section we turn to the case that $U$ and $V$ are inessential and provide a proof of Theorem \ref{T:main1}. We start with a lemma which shows that the results of Section \ref{S:comp} apply.

\begin{lemma}\label{L:Fred}
Let $U\in\fI(\cX)$ and $V\in\fI(\cY)$. Assume $U$ and $V$ are EAE with $E$ and $F$ as in \eqref{EFform1}. Then $F_{11}$, $F_{22}$, $E_{11}$ and $\what{E}_{11}$ are Fredholm and
\begin{equation}\label{IndexIds}
\Ind(F_{11})=-\Ind(F_{22}) \ands \Ind(E_{11})=-\Ind(\what{E}_{11}).
\end{equation}
\end{lemma}




\begin{proof}[\bf Proof]
By identities (i), (vii) and (v) in \eqref{EAEIds} we have $-F_{22}F_{11}=I-\whatE_{21}U$ and $-F_{11}F_{22}=I-E_{21}V$, and thus, by the definition of inessential operators, $-F_{11}F_{22}$ and $-F_{22}F_{11}$ are Fredholm. Furthermore, $\whatE_{21}U\in \fI(\cX)$ and $E_{21}V\in\fI(\cY)$, since $\fI(\cX)$ and $\fI(\cY)$ are ideals, so that $\Ind(F_{11}F_{22}) =\Ind(I-E_{21}V)=\Ind(I) =0$, and similarly $\Ind(F_{22}F_{11})=0$. Here we used that inessential perturbations do not effect the Fredholm index, which can be proved in the same way as for compact operators, using that the index function in locally constant. Hence $F_{11}F_{22}$ and $F_{22}F_{11}$ are Fredholm with index 0.

To see that $F_{11}$ and $F_{22}$ are Fredholm, one can use that $I+ F_{22}F_{11}=\whatE_{21}U$ and $I+F_{11}F_{22}=E_{21}V$ are inessential, in combination with an extension of a corollary of Atkinson's theorem \cite[Corollary~3.3.1]{A02} from compact to inessential operators. For a direct proof, note that identity (v) in \eqref{EAEIds} shows that within the Calkin algebra $\cB(\cY)/\fI(\cY)$ we have
$[F_{11}][F_{22}]=[F_{11}F_{22}]=[E_{21}V-I]=[I]$,
and similarly, by identities (i) and (vii) in \eqref{EAEIds}, $[F_{22}][F_{11}]=[I]$. Hence $[F_{11}]$ and $[F_{22}]$ are invertible in $\cB(\cY)/\fI(\cY)$, and thus $F_{11}$ and $F_{22}$ are Fredholm. The relations for the Fredholm indices follows from \cite[Theorem~IV.13.1]{TL80} which proves the claims regarding $F_{11}$ and $F_{22}$. For $E_{11}$ and $\whatE_{11}$ one can apply a parallel reasoning based in identities (viii) and (x) in \eqref{EAEIds}.
\end{proof}

We now employ the results of Section \ref{S:comp} to prove the main result of the paper.

\begin{proof}[\bf Proof of Theorem \ref{T:main1}]
As indicated in the paragraph following the statement of Theorem \ref{T:main1}, it remains to show that if $U$ and $V$ are EAE, they are also EAOE. So assume $U$ and $V$ are EAE, then by Lemma \ref{L:EFform} we may assume $E$ and $F$ are as in \eqref{EFform1}. Since $U$ and $V$ are inessential, by Lemma \ref{L:Fred}, the operators $E_{11}$ and $F_{22}$ in \eqref{EFform1} are Fredholm, and hence have complementable kernels and ranges. Then by Proposition \ref{P:ComplRed}, $U$ and $V$ are EAE with $\cX_0=\kr E_{11}$ and $\cY_0=\cH_2$. Furthermore, $\kr E_{11}$ and $\cH_2$ are finite dimensional, since $E_{11}$ and $F_{22}$ are Fredholm, hence $U$ and $V$ are EAE with extensions to finite dimensional Banach spaces. The fact that $\kr E_{11}$ and $\cH_2$ are finite dimensional also implies that one can either embed $\kr E_{11}$ into $\cH_2$, or conversely, depending on which has the largest dimension. Thus, by the second statement of Proposition \ref{P:ComplRed}, $U$ and $V$ are EAOE. In both cases the extensions are on finite dimensional Banach spaces.
\end{proof}


\begin{remark}
The proof of Proposition \ref{P:ComplRed} relies on the fact that the operators $U_{22}$ and $V_{22}$, as in Lemma \ref{L:UVform}, are invertible. In case $U$ and $V$ are inessential, this implies that $\dim \cH_2=\dim \cG_1$ and $\dim \kr F_{22}=\dim \kr E_{11}$, in particular, $\Ind(F_{22})=\Ind(E_{11})$. In combination with Lemma \ref{L:Fred}, this yields
\[
\Ind(F_{22})=\Ind(E_{11})=-\Ind(\wtilE_{11})=-\Ind(F_{11}).
\]
For $\wtilE_{11}$ and $F_{11}$, the last identity, i.e., $\Ind(\wtilE_{11})=\Ind(F_{11})$, can also be strengthened to the fact that the dimensions of the kernels of $\wtilE_{11}$ and $F_{11}$ are the same, as well as the dimensions of the complements of their ranges. This fact is obtained by interchanging the roles of $U$ and $V$ and working with $E^{-1}$ and $F^{-1}$ as in \eqref{EFform1}.

The proof of Lemma \ref{L:InvSuf} also reveals that the index of the four Fredholm operators indicates the dimension of the finite dimensional extension in the EAOE of $U$ and $V$, with $\Ind(F_{22})$ being positive meaning one has to extend $U$ and $\Ind(F_{22})$ being negative meaning one has to extend $V$.
\end{remark}

\section{An application to multiplication operators}\label{S:Appl}

In this section we present an application of Theorem 2.5 of \cite{tHMR15} to multiplication operators on $L^p$-spaces on the complex unit circle $\BT$ for $1\leq p<\infty$. See \cite{BS06} for definitions and further details. With a function $f\in L^\infty$, i.e., a measurable and essentially bounded function on $\BT$, we associate its multiplication operator
\[
M_f\colon L^p\to L^p,\quad (M_f g)(e^{it})=f(e^{it}) g(e^{it}) \quad (\mbox{e.a.}\ t\in[-\pi,\pi]).
\]
Let $H^p$ denote the Hardy space of functions in $L^p$ with negatively indexed Fourier coefficients equal to zero. Write $P$ for the Riesz projection that projects $L^p$ onto $H^p$ and set $Q=I-P$ and $K^p=Q L^p$. Decompose $M_f$ as
\begin{equation}\label{Mfdec}
M_f=\mat{cc}{\wtilT_f & \wtilH_f\\  H_f & T_f}\colon\mat{c}{K^p\\ H^p}\to\mat{c}{K^p\\ H^p}.
\end{equation}
Then $H_f$ and $T_f$ are the Hankel and Toepitz operators associated with $f$, respectively, while via the usual identification of $K^p$ and $H^p$, the operators $\wtilH_f$ and $\wtilT_f$ can be identified with the Hankel and Toepitz operators associated with $\wtil{f}$, respectively, where $\wtil{f}$ is given by $\wtil{f}(z)=\overline{f(\overline{z})}$, $z\in \BT$.

If $f$ is in the Wiener algebra $\cW$, i.e., the Fourier coefficients of $f$ are absolutely summable, then $f$ is continuous on $\BT$ (and thus in $L^\infty$) and $H_f$ and $\wtil{H}_f$ are compact. Furthermore, by Wiener's $1/f$ Theorem, $1/f\in \cW$ if and only if $f(z)\neq 0$ for all $z\in\BT$. In particular, $1/f$ is in $L^\infty$, so that $M_{1/f}$ is well defined, and the Hankel operators $H_{1/f}$ and $\wtilH_{1/f}$ are compact.

\begin{theorem}\label{T:appl}
Let $f\in \cW$ such that $f(z)\neq 0$ for all $z\in\BT$. Then $H_f$ and $H_{1/f}$ generate the same operator ideal. Likewise, $\wtilH_f$ and $\wtilH_{1/f}$ generate the same operator ideal.
\end{theorem}

\begin{proof}
By our assumption $1/f\in \cW$, and thus $1/f\in L^\infty$. Then $M_f$ is invertible on $L^p$ with inverse
\[
M_f^{-1}=M_{1/f}=\mat{cc}{\wtilT_{1/f} & \wtilH_{1/f}\\  H_{1/f} & T_{1/f}}\colon \mat{c}{K^p\\ H^p}\to\mat{c}{K^p\\ H^p}.
\]
Reordering rows and columns gives
\[
\mat{cc}{\wtilH_f & \wtilT_f\\  T_f & H_f}^{-1}
= \mat{cc}{H_{1/f} & T_{1/f}\\ \wtilT_{1/f} & \wtilH_{1/f}}.
\]
This shows that $H_f$ and $H_{1/f}$ are MC, and thus EAE. As a consequence of Theorem 2.5 of \cite{tHMR15} we find that $H_f$ and $H_{1/f}$ generate the same operator ideal. By another reorganization of the rows and columns we find that  $\wtilH_f$ and $\wtilH_{1/f}$ are MC, and thus generate the same operator ideal.
\end{proof}

Let $B_p^\al$ denote the Besov space
\[
B_p^\al:=\left\{ g\in L^p \colon \int_{-\pi}^\pi |t|^{-1-\al p} \|\De_t^n g\|_p^p \tu{d}t <\infty \right\}
\]
where $\| . \|_p$ denotes the $L^p$-norm, $n$ is any integer greater than $\al$, $(\De_t g)(e^{is})=g(e^{i(s+t)})$, $s,t\in [-\pi,\pi]$, $g\in L^p$, and $\De_t^{k+1}=\De_t \De_t^{k}$. By a theorem of Peller \cite{P82}, cf., \cite[Theorem 10.9]{BS06}, for a function $f\in L^\infty$ we have $P f \in B_p^{1/p}$ if and only if $H_f$ is in the Schatten-von Neumann class $\fC_p$, i.e., if and only if the approximation numbers of $H_f$ are $p$-summable. Since $\fC_p$ forms an operator ideal, Theorem \ref{T:appl} yields the following corollary.

\begin{corollary}\label{C1}
Let $f\in \cW$ such that $f(z)\neq 0$ for all $z\in\BT$. Then $P f \in B_p^{1/p}$ if and only if $P (1/f) \in B_p^{1/p}$ and $P \wtil{f} \in B_p^{1/p}$ if and only if $P (1/\wtil{f}) \in B_p^{1/p}$.
\end{corollary}

Let $\cW^+=\cW\cap H^\infty$. By another application of Wiener's $1/f$ theorem, for $f\in \cW^+$ we have $1/f \in \cW^+$ if and only if $f(z)\neq 0$ for all $z$ in the closed complex unit disk $\overline{\BD}$. This leads to the following corollary.

\begin{corollary}\label{C2}
Let $f\in \cW^+$ such that $f(z)\neq 0$ for all $z\in\overline{\BD}$. Then $H_f$ and $H_{1/f}$ generate the same operator ideal. In particular, $f \in B_p^{1/p}$ if and only if $1/f \in B_p^{1/p}$.
\end{corollary}

Theorem \ref{T:appl} easily extends to matrix- and operator-valued functions. We also expect this result, and possibly the corollaries, to carry over to other domains such as the unit sphere in $\BC^d$, cf., \cite{FX09}, and Wiener-Hopf and Toeplitz integral operators on the half line, c.f., \cite{P88,P03,BS06}, but will not work out the details here.

\begin{remark}\label{R:EquivAS}
In case $p=2$ more can be concluded from the fact that $H_f$ and $H_{1/f}$ are EAE. Indeed, let $p=2$ and let $\al_n \downarrow 0$ and $\be_n\downarrow0$ be the sequences of singular values of the Hankel operators $H_f$ and $H_{1/f}$, respectively, multiplicity taken into account. By Theorem 6.3 in \cite{T14}, $H_f$ and $H_{1/f}$ being EAE implies that the sequences $(\al_n)$ and $(\be_n)$ are comparable after a shift, that is, there exists a positive integer $k$ and a $c>0$ such that
\[
c< \frac{\al_n}{\be_{n+k}} <1/c\ \ (n\in\BN)\quad \mbox{or}\quad
c< \frac{\be_n}{\al_{n+k}} <1/c\ \ (n\in\BN).
\]
For $p\neq 2$ it is at present not clear whether $H_f$ and $H_{1/f}$ EAE implies that their approximation numbers (or other $s$-numbers) must be comparable after a shift. However, in Proposition 3.3 of \cite{tHMR15} it is shown that there is a positive integer $k$ such that uniform lower bounds for $\al_n/\be_{n+k}$ and  $\be_n/\al_{n+k}$ exist.
\end{remark}

\paragraph{\bf Acknowledgments}
This work is based on research supported in part by the National Research Foundation of South Africa and the DST-NRF Centre of Excellence in Mathematical and Statistical Sciences (CoE-MaSS). Any opinion, finding and conclusion or recommendation expressed in this material is that of the authors and the NRF does not accept any liability in this regard. Opinions expressed and conclusions arrived at are those of the author and are not necessarily to be attributed to the CoE-MaSS.

The authors thank Niels Laustsen for his colloquium talk at NWU on February 1, 2017, and the subsequent discussion in which he brought the topic of `Banach spaces with (very) few operators' to their attention.



\begin{thebibliography}{99}

\bibitem{A04}
P. Aiena, {\em Fredholm and local spectral theory, with applications to multipliers}, Kluwer Academic Publishers, Dordrecht, 2004.


\bibitem{AH11}
S.A. Argyros and R.G. Haydon, A hereditarily indecomposable $\cL_\infty$-space that solves the scalar-plus-compact problem, {\em Acta Math.} {\bf 206} (2011), 1-–54.

\bibitem{A02}
W. Arveson, {\em A Short Course on Spectral Theory}, Graduate Texts in Mathematics, 209. Springer-Verlag, New York, 2002.

\bibitem{BGK84}
H. Bart, I. Gohberg, and M.A. Kaashoek, The coupling method for solving integral equations, in: {\em Topics in operator theory systems and networks (Rehovot, 1983)}, pp.\ 39--73, {\em Oper.\ Theory Adv.\ Appl.} {\bf 12}, Birkh\"auser, Basel, 1984.

\bibitem{BGKR05}
H. Bart, I. Gohberg, M.A. Kaashoek, and A.C.M. Ran, Schur complements and state space realizations, {\em Linear Algebra Appl.} {\bf 399} (2005), 203--224.

\bibitem{BGKR08}
H. Bart, I. Gohberg, M.A. Kaashoek, and A.C.M. Ran, \emph{Factorization of matrix and operator functions: the state space method}, Oper.\ Theory Adv.\ Appl.\  \textbf{178}, Birkh\"auser Verlag, Basel, 2008.

\bibitem{BT94}
H. Bart and V.\'{E}. Tsekanovskii, Complementary Schur complements, {\em Linear Algebra Appl.} {\bf 197/198} (1994), 651--658.

\bibitem{BT92a}
H. Bart and V.\'{E}. Tsekanovskii, Matricial coupling and equivalence after extension, in: {\em Operator theory and complex analysis (Sapporo, 1991)}, pp.\ 143--160, {\em Oper.\ Theory Adv.\ Appl.} {\bf 59}, Birkh\"auser, Basel, 1992.

\bibitem{BT92b}
H. Bart and V.\'{E}. Tsekanovskii, Schur complements and strong versions of matricial coupling and equivalence after extension, {\em Report Series Econometric Institute Erasmus University Rotterdam},
Report {\bf 9262/A}, 1992.

\bibitem{BS06}
A. B\"{o}ttcher and B. Silbermann, {\em Analysis of Toeplitz operators. Second edition}, Springer Monographs in Mathematics, Springer-Verlag, Berlin, 2006.

\bibitem{CP16}
M.C. C\^amara and J.R. Partington, Spectral properties of truncated Toeplitz operators by equivalence after extension, {\em J.\ Math.\ Anal.\ Appl.} {\bf 433} (2016), 762–-784.

\bibitem{CParxiv}
M.C. C\^amara and J.R. Partington, Asymmetric truncated Toeplitz operators and Toeplitz operators with matrix symbol, preprint, arXiv:1504.06446.

\bibitem{CDS14}
L.P. Castro, R. Duduchava, and F.-O.  Speck, Diffraction from polygonal-conical screens, an operator approach, {\em Oper.\ Theory Adv.\ Appl.} {\bf 242} (2014), 113-–137.

\bibitem{DKKKL13}
H.G. Dales, K. Kania, T. Kochanek, P. Koszmider, and N.J. Laustsen, Maximal left ideals of the Banach algebra of bounded operators on a Banach space, {\em Studia Math.} {\bf 218} (2013), 245-–286.

\bibitem{ETarxiv}
Christian Engstr\"om and Axel Torshage, On equivalence and linearization of operator matrix functions with unbounded entries, preprint, arXiv:1612.01373.

\bibitem{FX09}
Q. Fang and J. Xia, Schatten class membership of Hankel operators on the unit sphere, {\em J. Funct.\ Anal.} {\bf 257} (2009), 3082–-3134.

\bibitem{GGK90}
I. Gohberg, S. Goldberg, and M.A. Kaashoek, {\em Classes of linear operators.\ Vol.\ I}, Oper.\ Theory Adv.\ Appl.\ {\bf 49}, Birkh\"auser Verlag, Basel, 1990.

\bibitem{GT63}
S. Goldberg and E.O. Thorp, On some open questions concerning strictly singular operators, {\em Proc.\ Amer.\ Math.\ Soc.} {\bf 14} (1963), 334-–336.

\bibitem{G94}
M. Gonz\'{a}lez, On essentially incomparable Banach spaces, {\em Math.\ Z.} {\bf 215} (1994), 621-–629.

\bibitem{G07}
M. Gonz\'{a}lez, Banach spaces with small Calkin algebras, in:\ {\em Perspectives in operator theory}, pp.\ 159–-170, Banach Center Publ., {\bf 75}, Polish Acad.\ Sci.\ Inst.\ Math., Warsaw, 2007.

\bibitem{GM93}
W.T. Gowers and B. Maurey, The unconditional basic sequence problem, {\em J. Amer.\ Math.\ Soc.} {\bf 6} (1993), 851-–874.

\bibitem{tHMR15}
S. ter Horst, M. Messerschmidt, and A.C.M. Ran, Equivalence after extension for compact operators on Banach spaces, {\em J. Math.\ Anal.\ Appl.} {\bf 431} (2015), 136–-149.

\bibitem{tHR13}
S. ter Horst and A.C.M. Ran, Equivalence after extension and matricial coupling coincide with Schur coupling, on separable Hilbert spaces, \emph{Linear Algebra Appl.} {\bf 439} (2013), 793--805.



\bibitem{K63}
D. Kleinecke, Almost-finite, compact, and inessential operators, {\em Proc.\ Amer.\ Math.\ Soc.} {\bf 14} (1963), 863–-868.

\bibitem{LT77}
J. Lindenstrauss and L. Tzafriri, {\em Classical Banach spaces. I. Sequence spaces}, Ergebnisse der Mathematik und ihrer Grenzgebiete {\bf 92}, Springer-Verlag, Berlin-New York, 1977.

\bibitem{M03}
B. Maurey, Banach spaces with few operators, in:\ {\em Handbook of the geometry of Banach spaces, Vol.\ 2}. pp.\  1247-–1297, North-Holland, Amsterdam, 2003.

%

\bibitem{P88}
J.R. Partington, {\em An introduction to Hankel operators}, London Mathematical Society Student Texts {\bf 13}, Cambridge University Press, Cambridge, 1988.

\bibitem{P82}
V.V. Peller, Vectorial Hankel operators, commutators and related operators of the Schatten-von Neumann class $\gamma_p$, {\em Integral Equations Operator Theory}  {\bf 5} (1982), 244-–272.

\bibitem{P03}
V.V. Peller, Vladimir, {\em Hankel operators and their applications}, Springer Monographs in Mathematics, Springer-Verlag, New York, 2003.

\bibitem{R99}
C.J. Read, Strictly singular operators and the invariant subspace problem, {\em Studia Math.} {\bf 132} (1999), 203–-226.

\bibitem{S68}
M. Schechter, Riesz operators and Fredholm perturbations, {\em Bull.\ Amer.\ Math.\ Soc.} {\bf 74} (1968), 1139-–1144.

\bibitem{S17}
F.-O. Speck, Paired operators in asymmetric space setting, {\em Oper.\ Theory Adv.\ Appl.} {\bf 259} (2017), 681--702.

\bibitem{TL80}
A.E. Taylor and D.C. Lay, {\em Introduction to functional analysis. Second edition}, John Wiley \& Sons, New York-Chichester-Brisbane, 1980.

\bibitem{T14}
D. Timotin, Schur coupling and related equivalence relations for operators
on a Hilbert space, {\em Linear Algebra Appl.} {\bf 452} (2014), 106--119.

\bibitem{T08}
C. Tretter, {\em Spectral theory of block operator matrices and applications}, Imperial College Press, London, 2008.

\bibitem{Z05}
F. Zhang (ed.), {\em The Schur complement and its applications}, Numerical Methods and Algorithms {\bf 4}, Springer-Verlag, New York, 2005.


\end{thebibliography}
\end{document}